\documentclass[11pt,leqno]{article}
\usepackage{float}
\restylefloat{table*}
\usepackage{booktabs}

\usepackage{empheq}
\usepackage{amsmath, amssymb} 
\usepackage{theorem}	
\usepackage{mathpazo}
\usepackage[margin=3cm]{geometry}
\usepackage[title]{appendix}

\usepackage{caption,subcaption}

\usepackage{graphicx, xcolor, soul}
\definecolor{labelkey}{rgb}{0,0.08,0.45}
\definecolor{refkey}{rgb}{0,0.6,0.0}
\definecolor{Brown}{rgb}{0.45,0.0,0.05}
\definecolor{lime}{rgb}{0.00,0.8,0.0}
\definecolor{lblue}{rgb}{0.5,0.5,0.99}

\usepackage{ifpdf} 							
\ifpdf
	\usepackage{hyperref}
\else
	\usepackage[dvipdfm]{hyperref}
\fi

\definecolor{myblue}{rgb}{.9, .9, 1}
  \newcommand*\mybluebox[1]{%
    \colorbox{myblue}{\hspace{1em}#1\hspace{1em}}}

\newcommand{\sepp}{\setlength{\itemsep}{-2pt}}

\newcommand{\menge}[2]{\big\{{#1}~\big |~{#2}\big\}} 

\newcommand{\nnn}{\ensuremath{{n\in{\mathbb N}}}}
\newcommand{\ntoinf}{\ensuremath{{n\to\infty}}}
\newcommand{\ve}{\ensuremath{\varepsilon}}
\newcommand{\scal}[2]{\left\langle{#1},{#2}  \right\rangle}

\newcommand{\exi}{\ensuremath{\exists\,}}
\newcommand{\zeroun}{\ensuremath{\left]0,1\right[}}
\newcommand{\RR}{\ensuremath{\mathbb R}}
\newcommand{\RP}{\ensuremath{\mathbb{R}_+}}
\newcommand{\RPP}{\ensuremath{\mathbb{R}_{++}}}
\newcommand{\RMM}{\ensuremath{\mathbb{R}_{--}}}

\newcommand{\RX}{\ensuremath{\,\left]-\infty,+\infty\right]}}

\newcommand{\NN}{\ensuremath{\mathbb N}}
\newcommand{\dom}{\ensuremath{\operatorname{dom}}}

\newcommand{\reli}{\ensuremath{\operatorname{ri}}}
\newcommand{\inte}{\ensuremath{\operatorname{int}}}

\newcommand{\epi}{\ensuremath{\operatorname{epi}}}

\newcommand{\Fix}{\ensuremath{\operatorname{Fix}}}
\newcommand{\Id}{\ensuremath{\operatorname{Id}}}

\newcommand{\pinf}{\ensuremath{+\infty}}

\newtheorem{theorem}{Theorem}[section]

\newtheorem{corollary}[theorem]{Corollary}
\newtheorem{proposition}[theorem]{Proposition}
\newtheorem{definition}[theorem]{Definition}
\theoremstyle{plain}{\theorembodyfont{\rmfamily}
}
\theoremstyle{plain}{\theorembodyfont{\rmfamily}
}
\theoremstyle{plain}{\theorembodyfont{\rmfamily}
}
\theoremstyle{plain}{\theorembodyfont{\rmfamily}
\newtheorem{example}[theorem]{Example}}
\newtheorem{fact}[theorem]{Fact}
\theoremstyle{plain}{\theorembodyfont{\rmfamily}
\newtheorem{remark}[theorem]{Remark}}

\newenvironment{proof}[1][Proof]{\par
	\vspace{-12pt}
	\trivlist
	\item[\hskip\labelsep\itshape #1.]\ignorespaces
	}{%
	\hfill\ensuremath{\blacksquare}\endtrivlist
}

\allowdisplaybreaks % or locally if problems {\allowdisplaybreaks
\begin{document}

\title{Proximal point algorithm, Douglas--Rachford algorithm and alternating
projections: a case study}

\author{
Heinz H.\ Bauschke\thanks{
Mathematics, University of British Columbia, Kelowna, B.C.\ V1V~1V7, Canada. 
E-mail: \texttt{heinz.bauschke@ubc.ca}.},~
Minh N.\ Dao\thanks{
Department of Mathematics and Informatics, Hanoi National University of Education, 136 Xuan Thuy, Hanoi, Vietnam,
and Mathematics, University of British Columbia, Kelowna, B.C.\ V1V~1V7, Canada.
E-mail: \texttt{minhdn@hnue.edu.vn}.},~
Dominikus Noll\thanks{
Institut de Math{\'e}matiques, Universit{\'e} de Toulouse, 118 route de Narbonne, 31062 Toulouse, France.
E-mail: \texttt{noll@mip.ups-tlse.fr}.}~~~and~
Hung M.\ Phan\thanks{
Department of Mathematical Sciences, University of Massachusetts Lowell,
265 Riverside St., Olney Hall 428, Lowell, MA 01854, USA. 
E-mail: \texttt{hung\_phan@uml.edu}.}
}

\date{January 26, 2015}

\maketitle

\begin{abstract} \noindent
Many iterative methods for solving optimization or feasibility problems
have been invented, and often convergence of the iterates to some solution is proven.
Under favourable conditions, one might have additional bounds on the
distance of the iterate to the solution leading thus to 
\emph{worst case estimates}, i.e., how fast the algorithm must converge.

Exact convergence estimates are typically hard to come by.
In this paper, we consider the complementary problem of finding 
\emph{best case estimates}, i.e., how slow the algorithm has to converge,
and we also study 
\emph{exact asymptotic rates of convergence}. Our investigation focuses on convex feasibility
in the Euclidean plane, where one set is the real axis while the other
is the epigraph of a convex function. This case study allows us to obtain
various convergence rate results. We focus on the popular method of
alternating projections and the Douglas--Rachford algorithm. These methods
are connected to the proximal point algorithm which is also
discussed. Our findings suggest
that the Douglas--Rachford algorithm outperforms the method of alternating
projections in the absence of constraint qualifications. Various examples
illustrate the theory. 
\end{abstract}

{\small

}

\section{Introduction}

\subsection*{Three algorithms}

Let $X$ be a Euclidean space, with inner product $\scal{\cdot}{\cdot}$ and
induced norm $\|\cdot\|$, and let $f\colon X\to\RX$ be convex, lower
semicontinuous, and proper. A classical method for finding a minimizer of
$f$ is the \emph{proximal point algorithm (PPA)}. 
It requires using the \emph{proximal point mapping} (or proximity operator)
which was pioneered by Moreau \cite{Moreau}:

\begin{fact}[proximal mapping] 
\label{f:prox}
For every $x\in X$, there exists a unique point
$p=P_f(x)\in X$ such that 
$\min_{y\in X} f(y) + \tfrac{1}{2}\|x-y\|^2 =
f(p)+\tfrac{1}{2}\|x-p\|^2$.
The induced operator $P_f\colon X\to X$ is firmly
nonexpansive\footnote{
Note that if $f=\iota_C$ is the \emph{indicator function} of a nonempty
closed convex subset of $X$, then $P_f = P_C$, where the $P_C$ is the
nearest point mapping or \emph{projector} of $C$; the corresponding
\emph{reflector} is $R_C = 2P_C-\Id$.}, i.e., 
$(\forall x\in X)(\forall y \in X)$
$\|P_f(x)-P_f(y)\|^2 + \|(\Id-P_f)x-(\Id-P_f)y\|^2 \leq \|x-y\|^2$. 
\end{fact}

The proximal point algorithm was proposed by Martinet \cite{Martinet}
and further studied by Rockafellar \cite{Rockprox}. 
Nowadays numerous extensions exist; however, here we focus only
on the most basic instance of PPA:

\begin{fact}[proximal point algorithm (PPA)]
\label{f:ppa}
Let $f\colon X\to\RX$ be convex, lower semicontinuous, and
proper. Suppose that $Z$, the set of minimizers of $f$, is
nonempty, and let $x_0\in X$. Then the sequence
generated by
\begin{equation}
(\forall\nnn)\quad x_{n+1} = P_f(x_n)
\end{equation}
converges to a point in $Z$ and it satisfies
\begin{equation}
(\forall z\in Z)(\forall\nnn) \quad
\|x_{n+1}-z\|^2 + \|x_n-x_{n+1}\|^2 \leq \|x_n-z\|^2.
\end{equation}
\end{fact}

An ostensibly quite different type of optimization problem is, 
for two given closed convex nonempty subsets $A$ and $B$ of $X$,
to find a point in $A\cap B\neq\varnothing$. 
Let us present two fundamental algorithms for solving this convex
feasibility problem. The first method was proposed by Bregman
\cite{Bregman}. 

\begin{fact}[method of alternating projections (MAP)]
\label{f:map}
Let $a_0\in A$ and set
\begin{equation}
(\forall\nnn)\quad a_{n+1} = P_AP_B(a_n). 
\end{equation}
Then $(a_n)_\nnn$ converges to a point $a_\infty\in C = A\cap B$. 
Moreover, 
\begin{equation}
(\forall c\in C)(\forall\nnn)\quad
\|a_{n+1}-c\|^2 + \|a_{n+1}-P_Ba_n\|^2 + \|P_Ba_n-a_n\|^2 \leq
\|a_n-c\|^2.
\end{equation}
\end{fact}

The second method is the celebrated Douglas--Rachford algorithm. 
The next result can be deduced by combining \cite{LM} and \cite{josa}. 

\begin{fact}[Douglas--Rachford algorithm (DRA)]
\label{f:dra}
Set $T = \Id-P_A+P_BR_A$, 
let $z_0\in X$, and set
\begin{equation}
(\forall\nnn)\quad a_n = P_Az_n 
\;\;\text{and}\;\;
z_{n+1} = Tz_n.
\end{equation}
Then\footnote{Here $\Fix T = \menge{x\in X}{x=Tx}$ is the set of fixed
points of $T$, and $N_{A-B}(0)$ stands for the normal cone of the set $A-B
= \menge{a-b}{a\in A,b\in B}$ at $0$.}
$(z_n)_\nnn$ converges to some point in $z_\infty \in \Fix T = (A\cap
B)+N_{A-B}(0)$,
and $(a_n)_\nnn$ converges to $P_Az_\infty\in A\cap B$. 
\end{fact}

Again, there are numerous refinements and adaptations of MAP and
DRA; however, 
it is here not our goal to survey the most general results
possible\footnote{See, e.g., \cite{BC2011} for various more general
variants of PPA, MAP, and DRA.} but rather
to focus on the speed of convergence. We will make this precise in the
next subsection.

\subsection*{Goal and contributions}

Most rate-of-convergence results for PPA, MAP, and DRA
take the following form: \emph{If some additional condition is satisfied,
then the convergence of the sequence is \emph{at least as good as} 
some form of ``fast'' convergence (linear, superlinear, quadratic
etc.).} This can be interpreted as a \emph{worst case analysis}. 
In the generality considered here\footnote{Some results are known
for MAP when the sets are linear subspaces; however, the slow
(sublinear) convergence can only be observed in 
\emph{infinite-dimensional Hilbert space}; see \cite{DHsurvey}
and references therein.}, we are not aware of results that approach 
this problem from the other side, i.e., that address the question:
\emph{Under which conditions is the convergence \emph{no better than} some form
of ``slow'' convergence?} This concerns the \emph{best case analysis}. 
\emph{Ideally, one would like an \emph{exact asymptotic rate of
convergence} in the sense of \eqref{e:sim} below.}

While we do not completely answer these questions, we do set out
to tackle them by providing a \emph{case study} 
when $X=\RR^2$ is the
Euclidean plane, the set $A = \RR\times \{0\}$ is the real axis, and
the set $B$ is the epigraph of a proper lower semicontinuous convex
function $f$. We will see that in this case MAP and DRA
have connections to the PPA applied to $f$. We focus in particular on the case not
covered by conditions guaranteeing linear convergence
of MAP or DRA \footnote{Indeed, the
most common sufficient condition for linear convergence in either
case is $\reli(A)\cap\reli(B)\neq\varnothing$; 
see \cite[Theorem~3.21]{BLPW2} for MAP
and \cite{Hung} or \cite[Theorem~8.5(i)]{BNP} for DRA.}.
We originally expected the
behaviour of MAP and DRA in cases of ``bad geometry'' to be
similar\footnote{This expectation was founded in the similar behaviour of
MAP and DRA for two subspaces; see \cite{BBNPW}.}.
It came to us as surprise that this appears not to be the case. In fact, the examples
we provide below suggest that DRA performs significantly
better than MAP. Concretely, suppose that $B$ is the epigraph of the function 
$f(x) = (1/p)|x|^p$, where $1<p<\pinf$. Since $A = \RR\times\{0\}$, we have
that $A\cap B = \{(0,0)\}$ and since $f'(0)=0$, the ``angle'' between $A$
and $B$ at the intersection is $0$. As expected
MAP converges sublinearly (even logarithmically) to $0$.
However, DRA converges faster in all cases: superlinearly (when $1<p<2$),
linearly (when $p=2$) or logarithmically (when $2<p<\pinf$).
This example is deduced by general results we obtain on
\emph{exact} rates of convergence for PPA, MAP and DRA. 

\subsection*{Organization}

The paper is organized as follows.
In Section~\ref{s:aux}, we provide various auxiliary results on the convergence
of real sequences. These will make the subsequent analysis of PPA, MAP, and
DRA more structured. 
Section~\ref{s:ppa} focuses on the PPA. After reviewing results on finite,
superlinear, and linear convergence, we exhibit a case where the asymptotic
rate is only logarithmic. 
We then turn to MAP in Section~\ref{s:map} and provide results on the
asymptotic convergence. We also draw the connection between MAP and PPA and
point out that a result of G\"uler is sharp.
In Section~\ref{s:dra}, we deal with DRA, draw again a connection to PPA and 
present asymptotic convergence. 
The notation we employ is fairly standard and follows, e.g., \cite{Rocky} and \cite{BC2011}. 

\section{Auxiliary results}
\label{s:aux}

In this section we collect various results that facilitate the subsequent
analysis of PPA, MAP and DRA. 
We begin with the following useful result which appears to be part of the
folklore\footnote{
Since we were able to locate only an online reference,
we include a proof in Appendix~\ref{a:Stolz}.}. 

\begin{fact}[generalized Stolz--Ces\`aro theorem]
\label{f:genStolz}
Let $(a_n)_\nnn$ and 
$(b_n)_\nnn$ be sequences in $\RR$ such that
$(b_n)_\nnn$ is unbounded and either strictly monotone
increasing or strictly monotone decreasing. 
Then 
\begin{equation}
\label{e:genStolz}
\varliminf_{n\to\infty} \frac{a_{n+1}-a_n}{b_{n+1}-b_n}
\leq 
\varliminf_{n\to\infty} \frac{a_n}{b_n}
\leq 
\varlimsup_{n\to\infty} \frac{a_n}{b_n}
\leq 
\varlimsup_{n\to\infty} \frac{a_{n+1}-a_n}{b_{n+1}-b_n},
\end{equation}
where the limits may lie in $[-\infty,+\infty]$. 
\end{fact}

Setting $(b_n)_\nnn = (n)_\nnn$ in Fact~\ref{f:genStolz}, we
obtain the following:

\begin{corollary}
The following inequalities hold for an arbitrary sequence $(x_n)_\nnn$
in $\RR$:
\label{c:Cesaro}
\begin{equation}\label{e:cCesaro}
\varliminf_{n\to\infty} (x_{n+1}-x_n)
\leq 
\varliminf_{n\to\infty} \frac{x_n}{n}
\leq 
\varlimsup_{n\to\infty}\frac{x_n}{n}
\leq 
\varlimsup_{n\to\infty} (x_{n+1}-x_n).
\end{equation}
\end{corollary}

For the remainder of this section, we assume that
\begin{empheq}[box=\mybluebox]{equation}
\text{
$g\colon\RPP\to\RPP$ is increasing and
$H$ is an antiderivative of $-1/g$.}
\end{empheq}

\begin{example}[$x^q$]
\label{ex:x^q}
Let $g(x)= x^q$ on $\RPP$, where
$1\leq q<\infty$. 
If $q>1$, then 
$-1/g(x)= -x^{-q}$ and we can choose 
$H(x)= x^{1-q}/(q-1)$ which has the inverse
$H^{-1}(x) = 1/((q-1)x)^{1/(q-1)}$.
If $q=1$, then we can choose 
$H(x) = -\ln(x)$ which has the inverse $H^{-1}(x) = \exp(-x)$. 
\end{example}

\begin{proposition}
\label{p:utility}
Let $(\beta_n)_\nnn$ and $(\delta_n)_\nnn$ be sequences in $\RPP$,
and suppose that 
\begin{equation}
(\forall\nnn)\quad 
\beta_{n+1} = \beta_n-\delta_ng(\beta_n). 
\end{equation}
Then the following hold:
\begin{enumerate}
\item
\label{p:utility1}
\label{e:p-util-b}
$\displaystyle
(\forall\nnn)\quad 
\delta_n \leq H(\beta_{n+1})-H(\beta_n) \leq
\delta_{n+1}\frac{\beta_{n}-\beta_{n+1}}{\beta_{n+1}-\beta_{n+2}}
= \delta_{n}\frac{g(\beta_n)}{g(\beta_{n+1})}$.
\item
\label{p:utility2}
$\displaystyle
\varliminf_{n\to\infty} \delta_n \leq 
\varliminf_{n\to\infty} \frac{H(\beta_n)}{n} \leq 
\varlimsup_{n\to\infty} \frac{H(\beta_n)}{n}
\leq 
\varlimsup_{n\to\infty} \delta_{n}\frac{g(\beta_n)}{g(\beta_{n+1})}
$.
\item
\label{p:utility3}
\label{e:p-util-a}
If $(\delta_n)_\nnn$ is convergent, say
$\delta_n\to\delta_\infty$,
and $\frac{g(\beta_n)}{g(\beta_{n+1})} \to 1$,
then 
$\frac{H(\beta_n)}{n} \to \delta_\infty$. 
\end{enumerate}
\end{proposition}
\begin{proof}
For every $\nnn$, we have
\begin{subequations}
\begin{align}
\delta_n = \frac{\beta_n-\beta_{n+1}}{g(\beta_n)}
&\leq \int_{\beta_{n+1}}^{\beta_n} \frac{dx}{g(x)} 
= H(\beta_{n+1})-H(\beta_n)\\
&\leq \frac{\beta_n-\beta_{n+1}}{g(\beta_{n+1})}
=\delta_{n+1}\frac{\beta_n-\beta_{n+1}}{\beta_{n+1}-\beta_{n+2}}
=\delta_n \frac{g(\beta_n)}{g(\beta_{n+1})}.
\end{align}
\end{subequations}
Hence \ref{p:utility1} holds.
Combining with \eqref{e:cCesaro},
we obtain \ref{p:utility2}.
Finally, \ref{p:utility3} follows from \ref{p:utility2}. 
\end{proof}

\begin{corollary}
\label{c:utility}
Let $(x_n)_\nnn$ and $(\delta_n)_\nnn$ be sequences in $\RPP$
such that 
\begin{equation}
(\forall\nnn)\quad x_n = x_{n+1}+\delta_n g(x_{n+1}). 
\end{equation}
Then the following hold:
\begin{enumerate}
\item
\label{c:utility1}
$\displaystyle
(\forall\nnn)\quad 
\delta_n \frac{g(x_{n+1})}{g(x_n)} 
\leq H(x_{n+1})-H(x_n) \leq
\delta_{n}$.
\item
\label{c:utility2}
$\displaystyle
\varliminf_\ntoinf \delta_n \frac{g(x_{n+1})}{g(x_n)} 
\leq 
\varliminf_\ntoinf \frac{H(x_n)}{n} \leq 
\varlimsup_\ntoinf \frac{H(x_n)}{n}
\leq 
\varlimsup_\ntoinf \delta_{n}
$.
\item
\label{c:utility3}
If $(\delta_n)_\nnn$ is convergent, say $\delta_n\to\delta_\infty$,
and  $\frac{g(x_{n+1})}{g(x_{n})} \to 1$, 
then $\frac{H(x_n)}{n} \to \delta_\infty$.
\end{enumerate}
\end{corollary}
\begin{proof}
Indeed, 
set $(\forall\nnn)$ $\ve_n = \delta_n \frac{g(x_{n+1})}{g(x_n)}$
and rewrite 
the update 
\begin{equation}
x_{n+1} = x_n -\delta_n \frac{g(x_{n+1})}{g(x_n)} g(x_n) = x_n -
\ve_n g(x_n).
\end{equation}
Now apply Proposition~\ref{p:utility}. 
\end{proof}

\begin{definition}[types of convergence]
Let $(\alpha_n)_\nnn$ be a sequence in $\RPP$ such that
$\alpha_n\to 0$,
and suppose there exist $1\leq q <\pinf$ such that
\begin{equation}
\frac{\alpha_{n+1}}{\alpha_n^q} \to c\in\RP. 
\end{equation}
Then the convergence of $(\alpha_n)_\nnn$ to $0$ is:
\begin{enumerate}
\item
\emph{with order $q$} if $q>1$ and $c>0$;
\item 
\emph{superlinear} if $q=1$ and $c=0$;
\item
\emph{linear} if $q=1$ and $0<c<1$;
\item
\emph{sublinear} if $q=1$ and $c=1$;
\item
\emph{logarithmic} if it is sublinear and
$|\alpha_{n+1}-\alpha_{n+2}|/|\alpha_{n} - \alpha_{n+1}| \to 1$.
\end{enumerate}
If $(\beta_n)_\nnn$ is also a sequence in $\RPP$,
it is convenient to define
\begin{equation}
\label{e:sim}
\alpha_n\sim\beta_n
\quad\Leftrightarrow\quad
\lim_\ntoinf \frac{\alpha_n}{\beta_n} \in\RPP.
\end{equation}
\end{definition}

The following example exhibits a case where we obtain
a simple exact asymptotic rate of convergence. 

\begin{example}
\label{ex:scheissruecken}
Let $(x_n)_\nnn$ and $(\delta_n)_\nnn$ be sequences in $\RPP$,
and let $1<q<\infty$.
Suppose that  
\begin{equation}
\delta_n\to\delta_\infty\in\RPP,
\;\;
\frac{x_n}{x_{n+1}}\to 1,
\;\;\text{and}\;\;
(\forall\nnn)\;
x_n= x_{n+1}+\delta_n x_{n+1}^q. 
\end{equation}
Then $x_n\to 0$ logarithmically, 
\begin{equation}\label{e:2.7-1}
\frac{x_n}{\Big(\frac{1}{n}\Big)^{1/(q-1)}}\to 
\frac{1}{((q-1)\delta_\infty)^{1/(q-1)}}
\quad\text{and}\quad
x_n \sim \Big(\frac{1}{n}\Big)^{1/(q-1)}.
\end{equation}
\end{example}
\begin{proof}
Suppose that $g(x) = x^q$ and  note that
$g(x_{n+1})/g(x_n)= (x_{n+1}/x_n)^q \to 1^q =1$.
This implies that $x_n\to 0$ logarithmically. 
Finally, \eqref{e:2.7-1} follows from Example~\ref{ex:x^q},
Corollary~\ref{c:utility}, and \eqref{e:sim}. 
\end{proof}

We conclude this section with some one-sided versions which are
useful for obtaining information about how fast or slow a sequence
must converge.

\begin{corollary}
\label{c:1-sd}
Let $(\beta_n)_\nnn$ and $(\rho_n)_\nnn$ be sequences in $\RPP$,
and suppose that 
\begin{equation}
(\forall\nnn)\;\beta_{n+1} \leq \beta_n-\rho_n g(\beta_n)
\;\;\text{and}\;\;
\underline{\rho} = \varliminf_\ntoinf \rho_n \in\RPP.
\end{equation}
Then
\begin{equation}
\big(\forall \ve \in\;]0,\underline{\rho}[\big)(\exi
m\in\NN)(\forall n\geq m)\;\;
\beta_n\leq H^{-1}\big( n(\underline{\rho}-\ve)\big).
\end{equation}
\end{corollary}
\begin{proof}
Observe that 
\begin{equation}
(\forall\nnn) \;\; \beta_{n+1} = \beta_n - \delta_n g(\beta_n),
\quad\text{where}\quad
\delta_n = \frac{\beta_n-\beta_{n+1}}{g(\beta_n)} \geq \rho_n.
\end{equation}
Hence, by Proposition~\ref{p:utility},
$\underline{\rho} \leq \varliminf_\ntoinf H(\beta_n)/n$.
Let $\ve \in]0,\underline{\rho}[$. 
Then there exists $m\in\NN$ such that 
$(\forall n\geq m)$ $\underline{\rho}-\ve \leq H(\beta_n)/n$
$\Leftrightarrow$
$H^{-1}( n(\underline{\rho}-\ve)) \geq \beta_n$. 
\end{proof}

\begin{example}
\label{ex:jonetal}
Let $(\beta_n)_\nnn$ and $(\rho_n)_\nnn$ be sequences in $\RPP$,
let $1\leq q < \infty$, 
and suppose that 
\begin{equation}
(\forall\nnn)\;\beta_{n+1} \leq \beta_n-\rho_n \beta_n^q
\;\;\text{and}\;\;
\underline{\rho} = \varliminf_\ntoinf \rho_n \in\RPP.
\end{equation}
Let $0<\ve<\underline{\rho}$.
Then there exists $m\in\NN$ such that the following hold:
\begin{enumerate}
\item
\label{ex:jonetal1}
If $q>1$, then
$\displaystyle (\forall n\geq m)\;\;
\beta_n \leq \frac{1}{\big((q-1)n(\underline{\rho}-\ve)\big)^{1/(q-1)}}
= O\big(1/n^{1/(q-1)}\big)$.
\item
If $q=1$, then 
$(\forall n\geq m)$ $\beta_n\leq\gamma^n$, where 
$\gamma = \exp(\ve-\underline{\rho})\in\zeroun$. 
\end{enumerate}
Consequently, the convergence of $(\beta_n)_\nnn$ to $0$ is at
least sublinear if $q>1$ and at least linear if $q=1$. 
\end{example}
\begin{proof}
Combine Example~\ref{ex:x^q} with Corollary~\ref{c:1-sd}.
\end{proof}

\begin{remark}
Example~\ref{ex:jonetal}\ref{ex:jonetal1}
can also be deduced from 
\cite[Lemma~4.1]{BLY}; see also \cite{AR}. 
\end{remark}

\begin{corollary}
\label{c:1-sd2}
Let $(\beta_n)_\nnn$ and $(\rho_n)_\nnn$ be sequences in $\RPP$,
and suppose that 
\begin{equation}
(\forall\nnn)\;\beta_n\ \geq \beta_{n+1} \geq \beta_n-\rho_n g(\beta_n)
\;\;\text{and}\;\;
\overline{\rho} = \varlimsup_\ntoinf \rho_n
\frac{g(\beta_n)}{g(\beta_{n+1})}\in\RP.
\end{equation}
Then
\begin{equation}
(\forall \ve \in\RPP)(\exi
m\in\NN)(\forall n\geq m)\;\;
\beta_n\geq H^{-1}\big( n(\overline{\rho}+\ve)\big).
\end{equation}
\end{corollary}
\begin{proof}
Observe that 
\begin{equation}
(\forall\nnn)\;\;
\beta_{n+1} = \beta_n - \delta_n g(\beta_n),
\quad\text{where}\quad
\delta_n = \frac{\beta_n-\beta_{n+1}}{g(\beta_n)} \leq \rho_n.
\end{equation}
Hence, by Proposition~\ref{p:utility},
$\varlimsup_\ntoinf H(\beta_n)/n \leq 
\overline{\rho}$.
Let $\ve\in\RPP$. Then there exists $m\in\NN$ such that 
$(\forall n\geq m)$ $\overline{\rho}+\ve \geq H(\beta_n)/n$
$\Leftrightarrow$
$H^{-1}( n(\overline{\rho}+\ve)) \leq \beta_n$. 
\end{proof}

\begin{example}
\label{ex:anotherjon}
Let $(\beta_n)_\nnn$ and $(\rho_n)_\nnn$ be sequences in $\RPP$,
let $1\leq q < \infty$, 
and suppose that 
\begin{equation}
(\forall\nnn)\;\beta_n\geq\beta_{n+1} \geq \beta_n-\rho_n \beta_n^q
\;\;\text{and}\;\;
\overline{\rho} = \varlimsup_\ntoinf
\rho_n\frac{\beta_n^q}{\beta_{n+1}^q} \in\RP.
\end{equation}
Let $\ve\in\RPP$.
Then there exists $m\in\NN$ such that the following hold:
\begin{enumerate}
\item
If $q>1$, then
$\displaystyle (\forall n\geq m)\;\;
\beta_n \geq
\frac{1}{\big((q-1)n(\overline{\rho}+\ve)\big)^{1/(q-1)}}$.
\item
If $q=1$, then 
$(\forall n\geq m)$ $\beta_n\geq\gamma^n$, where 
$\gamma = \exp(-\overline{\rho}-\ve)\in\zeroun$. 
\end{enumerate}
Consequently, the convergence of $(\beta_n)_\nnn$ to $0$ is at
best sublinear if $q>1$ and at best linear if $q=1$. 
\end{example}
\begin{proof}
Combine Example~\ref{ex:x^q} with Corollary~\ref{c:1-sd2}.
\end{proof}

\section{Proximal point algorithm (PPA)}

\label{s:ppa}

This section focuses on the proximal point algorithm. We assume that
\begin{empheq}[box=\mybluebox]{equation}
f \colon \RR\to\RX
\quad
\text{is convex, lower semicontinuous, proper,}
\end{empheq}
with 
\begin{empheq}[box=\mybluebox]{equation}
\text{
$f(0)=0$ and $f(x)>0$ when $x\neq 0$.
}
\end{empheq}
Given $x_0\in\RR$, we will study the basic proximal point iteration
\begin{empheq}[box=\mybluebox]{equation}
(\forall\nnn)\quad 
x_{n+1} = P_f(x_n). 
\end{empheq}
Note that if $x>0$ and $y<0$,
then $f(y)+\tfrac{1}{2}|x-y|^2 > f(0) + \tfrac{1}{2}|x-0|^2 \geq
f(P_fx)+\tfrac{1}{2}|x-P_fx|^2$. Hence
the behaviour of $f|_{\RMM}$ is \emph{irrelevant} for the determination of
$P_f|_{\RPP}$ (and an analogous statement holds for the
determination of $P_f|_{\RMM}$)!
For this reason, we restrict our attention to the case when
\begin{empheq}[box=\mybluebox]{equation}
x_0\in\RPP
\end{empheq}
is the starting point of the proximal point algorithm. 
The general theory (Fact~\ref{f:ppa}) then yields
\begin{equation}
\label{e:0703a}
x_0 \geq x_1 \geq \cdots \geq x_n \downarrow 0.
\end{equation}
In this section, it will be convenient to additionally assume that
\begin{empheq}[box=\mybluebox]{equation}
\label{e:feven}
\text{$f$ is an even function;}
\end{empheq}
although, as mentioned, the behaviour of $f|_{\RMM}$ is actually
irrelevant because $x_0\in\RPP$. 
Combining the assumption that $0$ is the unique minimizer of $f$
with \cite[Theorem~24.1]{Rocky},
we learn that 
\begin{equation}\label{e:f'(0)}
0\in\partial f(0) = 
 [f'_-(0),f'_+(0)] \cap \RR =
[-f'_+(0),f'_+(0)] \cap \RR.
\end{equation}

We start our exploration by discussing convergence in finitely
many steps.

\begin{proposition}[finite convergence]
\label{p:proxfinconv}
We have $x_n\to 0$ in finitely many steps,
regardless of the starting point $x_0\in\RPP$,
if and only if 
\begin{equation}
0<f'_+(0),
\end{equation}
in which case $P_fx_n = 0$
$\Leftrightarrow$
$x_n \leq f'_+(0)$. 
\end{proposition}
\begin{proof}
Let $x>0$. Then
$P_fx = 0$ 
$\Leftrightarrow$
$x \in 0 + \partial f(0)$
$\Leftrightarrow$
$x\leq f'_+(0)$
by \eqref{e:f'(0)}. 

Suppose first that $f'_+(0)>0$.
Then, by \eqref{e:f'(0)},  $0\in\inte\partial f(0)$ and,
using \eqref{e:0703a}, there exists $n\in\NN$
such that $x_n \leq f'_+(0)$.
It follows that $x_{n+1} = x_{n+2} = \cdots = 0$.
(Alternatively, this follows from a much more general result of
Rockafellar; see \cite[Theorem~3]{Rockprox} and also
Remark~\ref{r:rockfin} below.)

Now assume that there exists $n\in\NN$ such that
$P_fx_n=0$ and $x_n>0$. By the above, 
$x_n \leq f'_+(0)$ and thus $f'_+(0)>0$.
\end{proof}

An extreme case occurs when $f'_+(0)=\pinf$ in
Proposition~\ref{p:proxfinconv}:

\begin{example}[$\iota_{\{0\}}$ and the projector]
Suppose that $f=\iota_{\{0\}}$.
Then $P_f = P_{\{0\}}$ and $(\forall n\geq 1)$ $x_n=0$. 
\end{example}

\begin{example}[$|x|^1$ and the thresholder]
Suppose that $f=|\cdot|$ in which case
$\partial f(0)=[-1,1]$ and $f_+(0)=1$.
Proposition~\ref{p:proxfinconv} guarantees finite convergence
of the PPA. 
Indeed, either a direct argument or  \cite[Example~14.5]{BC2011}
yields
\begin{equation}
P_f \colon x\mapsto \begin{cases}
x - \frac{x}{|x|}, &\text{if $|x| > 1$;}\\
0, &\text{otherwise,}
\end{cases}
\end{equation}
Consequently, $x_n=0$ if and only if
$n \geq \lceil x_0\rceil$.
\end{example}

\begin{remark}
\label{r:rockfin}
In \cite[Theorem~3]{Rockprox}, 
Rockafellar 
provided a very general \emph{sufficient} condition for finite
convergence of the PPA (which works actually for finding zeros of 
a maximally monotone operator defined on a Hilbert space). In our present setting, his
condition is 
\begin{equation}
0 \in \inte \partial f(0). 
\end{equation}
By Proposition~\ref{p:proxfinconv}, this is also a condition that is 
\emph{necessary} for finite convergence. 
\end{remark}

Thus, we assume from now on that $f'_+(0)=0$, or equivalently 
(since $f$ is even and by \eqref{e:f'(0)}), that
\begin{empheq}[box=\mybluebox]{equation}
f'(0)=0.
\end{empheq}
in which case finite convergence fails and thus 
\begin{equation}
\label{e:150116car}
x_0 > x_1 > \cdots > x_n \downarrow 0.
\end{equation}

We now have the following sufficient condition for linear convergence.
The proof is a refinement of the ideas of Rockafellar in \cite{Rockprox}. 

\begin{proposition}[sufficient condition for linear convergence]
\label{p:sharper}
Suppose that
\begin{equation}
\label{e:psharper-d}
\lambda = \varliminf_{x\downarrow 0} \frac{f(x)}{x^2} 
\in \left]0,\pinf\right].
\end{equation}
Then the following hold:
\begin{enumerate}
\item
\label{p:sharper1}
If $\lambda < \pinf$, then 
there exists 
$\alpha_0 \in \big[\frac{1}{2\lambda},\frac{1}{\lambda}\big]$
such that
\begin{equation}
\label{e:psharper-c}
(\forall\ve>0)(\exi m\in\NN)(\forall n\geq m)\quad 
|x_{n+1}| \leq 
\frac{\alpha_0}{\sqrt{1+\alpha_0^2(1+2\lambda-2\ve)}} |x_n|.
\end{equation}
\item
\label{p:sharper2}
If $\lambda = \pinf$, then 
\begin{equation}
\label{e:psharper-c+}
(\forall\alpha>0)(\forall\varepsilon>0)(\exi m\in\NN)(\forall n\geq m)\quad 
|x_{n+1}| \leq 
\frac{\alpha}{\sqrt{1+\alpha^2(1+2\lambda-\ve)}} |x_n|.
\end{equation}
\end{enumerate}
\end{proposition}
\begin{proof}
By \cite[Remark~4 and Proposition~7]{Rockprox}, there exists
$\alpha_0\in\big[\frac{1}{2\lambda},\frac{1}{\lambda}\big]$ 
such that $(\partial f)^{-1}$ is 
Lipschitz continuous at $0$ with every modulus $\alpha>\alpha_0$. 
Let $\alpha>\alpha_0$. Then there exists $\tau>0$ such that
\begin{equation}\label{e:psharper-a}
(\forall |x|<\tau)(\forall z\in(\partial f)^{-1}(x))
\quad |z|\leq\alpha|x|.
\end{equation}
Since $x_n\to 0$
by \cite[Theorem~2]{Rockprox} (or \eqref{e:150116car}), %we can assume that $|x_n|\leq\tau$ and $|x_n-x_{n+1}|\leq\tau$ for all $n$.
there exists $m\in\NN$ such that $(\forall n \geq m)$ $|x_n-x_{n+1}|\leq\tau$.
Let $n\geq m$. 
Noticing that $x_n\in (\Id +\partial f)(x_{n+1})$, we have
\begin{equation}
x_{n+1}\in(\partial f)^{-1}(x_n-x_{n+1}).
\end{equation}
It follows by \eqref{e:psharper-a} that
\begin{equation}\label{e:psharper-b}
|x_{n+1}|\leq\alpha|x_n-x_{n+1}|.
\end{equation}
Since $x_n-x_{n+1}\in\partial f(x_{n+1})$, we have
\begin{equation}\label{e:psharper-e}
\scal{x_n-x_{n+1}}{x_{n+1}} =\scal{x_n-x_{n+1}}{x_{n+1}-0} \geq f(x_{n+1})-f(0)=f(x_{n+1}).
\end{equation}
Now for every $\ve>0$, employing \eqref{e:psharper-d} and
increasing $m$ if necessary, we can and do assume that
\begin{equation}\label{e:psharper-f}
(\forall n \geq m)\quad \scal{x_n-x_{n+1}}{x_{n+1}}\geq \left(\lambda-\frac{\ve}{2}\right)|x_{n+1}|^2.
\end{equation}
Let $n\geq m$. 
Combining \eqref{e:psharper-b} and \eqref{e:psharper-f},
we obtain
\begin{subequations}
\begin{align}
|x_n|^2&=|x_{n+1}|^2+|x_n-x_{n+1}|^2
+2\scal{x_n-x_{n+1}}{x_{n+1}}\\
&\geq
|x_{n+1}|^2+\frac{1}{\alpha^{2}}|x_{n+1}|^2+(2\lambda-\ve)|x_{n+1}|^2\\
&=\Big(\frac{1+\alpha^2(1+2\lambda-\ve)}{\alpha^2}\Big)|x_{n+1}|^2.
\end{align}
\end{subequations}
This gives 
\begin{equation}
|x_{n+1}|\leq
\frac{\alpha}{\sqrt{1+\alpha^2(1+2\lambda-\ve)}}|x_{n}| 
\end{equation}
and hence \eqref{e:psharper-c+} holds. 
Now assume that $\lambda < \pinf$ so that $\alpha_0>0$. 
Since 
$\alpha\mapsto \frac{\alpha}{\sqrt{1+\alpha^2(1+2\lambda-\ve)}} =\frac{1}{\sqrt{\frac{1}{\alpha^2}+(1+2\lambda-\ve)}}$ 
is strictly increasing on $\RP$, we note 
that the choice 
$\alpha = \alpha_0/\sqrt{1-\ve\alpha_0^2} > \alpha_0$ yields 
\eqref{e:psharper-c}. 
\end{proof}

\begin{remark}
Assume that $f$ is differentiable on $U= \left]0,\delta\right[$, where $\delta\in\RPP$. 
Then $f'(x) >0$ on $U$.
Note that 
$\lim_{x\downarrow 0} \frac{f'(x)}{x} = f''_{+}(0)$.
Therefore, L'H\^opital's rule shows that 
if $f''_+(0)$ exists in $[0,\pinf]$, then 
\begin{equation}
\lambda = \tfrac{1}{2}f''_+(0)
\end{equation}
in \eqref{e:psharper-d}.
A sufficient condition for $\lambda$ to exist is to assume that 
the function $f(x)/x^2$ is monotone on $U$ which in turn
happens when $2f(x)-xf'(x)$ is either nonnegative or nonpositive on $U$ by
using the quotient rule. 
\end{remark}

Although we won't need it in the remainder of this paper,
we point out that the proof of Proposition~\ref{p:sharper} still
works in a more general setting leading to the following result:

\begin{corollary}
\label{c:proxlin}
Let $H$ be a real Hilbert space,
and let $f\colon H\to\RX$ be convex,
lower semicontinuous and proper such that
$0$ is the unique minimizer of $f$.
Assume also that 
\begin{equation}
\label{e:proxlin1}
\lambda = \varliminf_{0\neq x \to 0} \frac{f(x)}{\|x\|^2} 
\in \left]0,\pinf\right].
\end{equation}
Then there exists 
$\alpha_0\in \big[\frac{1}{2\lambda},\frac{1}{\lambda}\big]$
such that %if $\alpha  \alpha_0$, then
\begin{equation}
\label{e:proxlin2}
(\forall\alpha>\alpha_0)(\forall\ve>0)(\exi m\in\NN)(\forall n\geq m)\quad 
\|x_{n+1}\| \leq
\frac{\alpha}{\sqrt{1+\alpha^2(1+2\lambda-\varepsilon) }} \|x_n\|.
\end{equation}
If $\lambda<\pinf$, then eventually
\begin{equation}
\label{e:150114a}
\|x_{n+1}\|\leq \frac{\alpha_0}{\sqrt{1+\alpha_0}^2}\|x_n\|,
\end{equation}
a result which can
also be deduced from \cite[Theorem~2]{Rockprox}. 
\end{corollary}

We now discuss powers of the absolute value function. 

\begin{example}[$|x|^2$ and linear convergence]
Suppose that $f(x)=x^2$. Then $x+f'(x) = 3x$ and hence $P_f =
\frac{1}{3}\Id$. We see that the actual linear rate of
convergence of the PPA is 
\begin{equation} \frac{1}{3}.
\end{equation}
Now consider Proposition~\ref{p:sharper} and
Corollary~\ref{c:proxlin}. 
Then clearly $\lambda=1$ in \eqref{e:psharper-d} and hence 
$\alpha_0\in \big[\frac{1}{2},1\big]$.
In fact, since $\partial f = 2\Id$ and so $(\partial f)^{-1} =
\frac{1}{2}\Id$, we know that the tightest choice for $\alpha_0$ is
$\alpha_0=\frac{1}{2}$.
The linear rate obtained by \eqref{e:150114a} is
${(1/2)}/{\sqrt{1+(1/2)^2}}$, i.e., 
\begin{equation}
\frac{1}{\sqrt{5}}.
\end{equation}
Let us compare to the linear rate provided by
Proposition~\ref{p:sharper},
where, for every $\ve>0$, we obtain 
${(1/2)}/{\sqrt{1+(1/2)^2(1+2-2\ve)}} = 1/\sqrt{7-2\ve}$.
From the proof of Proposition~\ref{p:sharper}, we see that we can
here actually set
$\ve=0$; thus, the rate provided is 
\begin{equation}
\frac{1}{\sqrt{7}}.
\end{equation}
In summary, $1/\sqrt{7}$, the rate from
Proposition~\ref{p:sharper}, is better
than $1/\sqrt{5}$, which comes from \eqref{e:150114a}; 
however, even the former does not capture the true rate $1/3$. 
\end{example}

\begin{example}[$|x|^q$, where $1<q<2$, and superlinear
convergence]\ \\
\label{ex:xq}
Suppose that $f(x) = |x|^q$, where $1<q<2$.
Note that $\lambda=\pinf$ and thus $\alpha_0=0$ in
\eqref{e:psharper-c}. 
In passing, we point out that we cannot use 
$\alpha_0$ itself in \eqref{e:psharper-c} because
it would imply finite convergence which does not occur by
\eqref{e:150116car}.
Set $\phi(x) =x +qx^{q-1}$ and note that 
$(\forall\nnn)$ $x_n =\phi(x_{n+1})$. 
Now set also $\psi(x) =qx^{q-1}$,
and assume that a sequence $(\rho_n)_\nnn$ satisfies
$(\forall\nnn)$ $\rho_n =\psi(\rho_{n+1})$. 
The sequence $(\rho_n)_\nnn$ can be thought of as an
approximation of $(x_n)_\nnn$. It has the advantage that
the implicit recursion is invertible and solvable; indeed
one may verify by induction that 
\begin{equation}
(\forall\nnn)\quad 
\rho_n = \rho_0^{1/(q-1)^n} (1/q)^{(1/(q-1)^n -1)/(2-q)};
\end{equation}
Assume furthermore that $\rho_0 = x_0$ is sufficiently close to $0$. 
Since $\phi$ and $\psi$ are increasing and $\phi > \psi > 0$ on $\RPP$,
we deduce that $(\forall n\geq 1)$ $x_n < \rho_n$. 
Therefore,
\begin{equation}
(\forall\nnn)\;\;
\frac{x_n}{\rho_n} =\frac{\phi(x_{n+1})}{\psi(x_{n+1})} \left( \frac{x_{n+1}}{\rho_{n+1}} \right)^{q-1} 
> \left( \frac{x_{n+1}}{\rho_{n+1}} \right)^{q-1},
\end{equation}
which implies that
\begin{equation}
0 < \frac{x_n}{\rho_n} < \left( \frac{x_1}{\rho_1}
\right)^{1/(q-1)^{n-1}} \to 0 % \quad\text{as}\; n \to \infty 
\end{equation}
because $1/(q-1) >1$.
Let $\nnn$. It follows from $x_n = x_{n+1} +qx_{n+1}^{q-1}
>x_{n+1}$ and $1 <q <2$ that $x_{n+1}x_n^{q-1} <x_n x_{n+1}^{q-1}$,
and so
\begin{equation}
\frac{x_n}{x_{n-1}} = \frac{x_{n+1} +qx_{n+1}^{q-1}}{x_n +qx_n^{q-1}} <\frac{x_{n+1}^{q-1}}{x_n^{q-1}}.
\end{equation} 
This gives $\frac{x_n^{1/(q-1)}}{x_{n-1}^{1/(q-1)}} <\frac{x_{n+1}}{x_n}$;
hence, 
\begin{equation}
\frac{x_n}{x_{n-1}^{1/(q-1)}} <\frac{x_{n+1}}{x_n^{1/(q-1)}}.
\end{equation}
On the other hand, $x_n = x_{n+1} +qx_{n+1}^{q-1} >qx_{n+1}^{q-1}$,
which yields $(\frac{x_n}{q})^{1/(q-1)} >x_{n+1}$ and hence 
$\frac{x_{n+1}}{x_n^{1/(q-1)}} <q^{1/(1-q)}$. 
The sequence
$(x_{n+1}/(x_n^{1/(q-1)}))_\nnn$ is thus increasing and
bounded above, and so it converges to some $\mu > 0$. We obtain
that $x_n \to 0$ superlinearly with order $1/(q-1)$.
\end{example}

\begin{example}[$|x|^q$, where $2<q$, and logarithmic
convergence] \ \\
Suppose that $f(x) = |x|^q$, where $2<q<\pinf$.
Because $(\forall\nnn)$ 
$x_n = x_{n+1} + qx_{n+1}^{q-1}$, we have 
$x_n/x_{n+1} = 1+ q x_{n+1}^{q-2} \to 1$.
It thus follows from Example~\ref{ex:scheissruecken} that
$x_n\to 0$ logarithmically and 
\begin{equation}
\frac{x_n}
{\big(\frac{1}{n}\big)^{1/(q-2)}}\to
\frac{1}{((q-2)q)^{1/(q-2)}}.
\end{equation}
\end{example}

Let us summarize what we found out in the previous three examples
about the behaviour of the PPA applied to $|x|^q$:

\begin{table*}[H] \centering
\begin{tabular}{@{}r|l@{}} \toprule
$q$ & PPA convergence of $x_n\to 0$ for $f(x)=|x|^q$\\ \midrule
$1<q<2$ & superlinear with order $1/(q-1)$ \\
$q=2$ & linear with rate $1/3$ \\
$2<q<\pinf$& logarithmic\\
\bottomrule
 \end{tabular}
\end{table*}

\section{Method of Alternating Projections (MAP)}

\label{s:map}
We now turn to the method of alternating projections. 
As in Section~\ref{s:ppa}, we assume without loss of generality
that 
\begin{subequations}\label{e:fAB}
\begin{empheq}[box=\mybluebox]{equation}
f \colon \RR\to\RX
\quad\text{is convex, lower semicontinuous, and proper,}
\end{empheq}
with 
\begin{empheq}[box=\mybluebox]{equation}
\text{$f$ even, $f(0)=0$, $f>0$ otherwise, and $f'(0)=0$.}
\end{empheq}
Furthermore, we set 
\begin{empheq}[box=\mybluebox]{equation}
A = \RR \times \{0\}
\quad\text{and}\quad
B= \epi f.
\end{empheq}
\end{subequations}

The projection onto $A$ is very simple: 
\begin{equation}
P_A\colon \RR^2\to\RR^2\colon (x,r)\mapsto (x,0).
\end{equation}
We now turn to $P_B$.

\begin{fact}
\label{f:epi}
{\rm (See \cite[Proposition~9.18 and
Proposition~28.28]{BC2011})}
Let $(x,r)\in (\dom f \times \RR)\smallsetminus B$.
Then
$P_{B}(x,r) = (y,f(y))$,
where $y$ satisfies 
$x\in y + (f(y)-r)\partial f(y)$.
Moreover,
$r<f(y)$ and 
$(\forall z\in\dom f)$
$(z-y)(x-y)\leq (f(z)-f(y))(f(y)-r)$.
\end{fact}

\begin{corollary}
\label{c:epi}
Suppose that $f$ is differentiable at $0$, 
let $(x,r)\in (\dom f \times \RR) \smallsetminus B$, 
and set $(y,f(y)) = P_B(x,r)$.
% If $x=0$, then $y=0$.
% Assume $x\neq 0$. 
% Then $y\in \left]0,x\right[$ and
Then $y =0$ if $x =0$, and 
$y$ lies strictly between $x$ and $0$ otherwise.
Furthermore, 
$x^2+r^2 \geq y^2 + f(y)^2 + (x-y)^2 + (r-f(y))^2$.
\end{corollary}
\begin{proof}
We use Fact~\ref{f:epi}. 
Observe that $r<f(y)$ and hence that $f(y)-r>0$. 
Choosing $z=0$ gives
$-y(x-y)\leq -f(y)(f(y)-r) \leq 0$. 
If $x =0$, this implies $y^2 \leq 0$, and so $y =0$.
Assume that $x\neq 0$.
Then $y\neq 0$ since $y =0$ implies $x =0 +(f(0) -r)f'(0) =0$. 
We obtain $-f(y)(f(y)-r)<0$, and then $-y(x-y)<0$, i.e., $y(x-y)>0$.
It follows that $y\in\left]0,x\right[$ if $x >0$, 
and $y\in \left]x,0\right[$ if $x <0$.
Finally, since $P_B$ is firmly nonexpansive (see, e.g., 
\cite[Proposition~4.8]{BC2011}) and 
$P_B(0,0)=(0,0)$, we obtain 
\begin{equation}
\|(x,r)-(0,0)\|^2 \geq \|(y,f(y))-(0,0)\|^2 +
\|(x,r)-(y,f(y))\|^2,
\end{equation}
which completes the proof.
\end{proof}

We now turn to the sequence generated by the method of
alternating projections.
We assume without loss of generality that
\begin{subequations}
\label{e:mapseq}
\begin{empheq}[box=\mybluebox]{equation}
x_0\in\RPP\cap\dom f,\;\;
a_0 = (x_0,0)\in A,\;\;
\end{empheq}
and 
\begin{empheq}[box=\mybluebox]{equation}
(\forall\nnn)\;\;
a_{n+1} = P_AP_B(a_n) = (x_{n+1},0).
\end{empheq}
\end{subequations}

Combining Fact~\ref{f:map}, \eqref{e:fAB} and \eqref{e:mapseq},
we learn that 
\begin{equation}
x_0 > x_1 > \cdots > x_n \downarrow 0.
\end{equation}

We are ready for our first result on the lack of linear
convergence for MAP. 

\begin{theorem}
\label{t:MAP}
The following hold:
\begin{equation}
\label{e:150117a}
(\forall\nnn)\quad
x_{n+1}(x_{n+1}-x_n) + f^2(x_{n+1}) \leq 0,
\end{equation}
\begin{equation}
\label{e:map}
(\forall\nnn)\quad 
x_n = x_{n+1} + f(x_{n+1})x_{n+1}^*, \text{ where }
x_{n+1}^*\in\partial f(x_{n+1}), 
\end{equation}
and $x_n\to 0$ sublinearly, i.e., 
\begin{equation}
\label{e:sublinear}
\frac{x_{n+1}}{x_n}\to 1.
\end{equation}
If $f$ is differentiable on some interval $[0,\delta]$, where
$\delta\in\RPP$, 
and there exists $q\in\RR$ such that 
\begin{equation}\label{e:lim}
\lim_{x\downarrow 0} \frac{f(x)f'(x)}{x^q} = c_q \in \RPP,
\end{equation}
then 
$x_n\to 0$ logarithmically, i.e., 
\begin{equation}
\frac{x_{n+1} -x_{n+2}}{x_{n} -x_{n+1}} \to 1.
\end{equation}
\end{theorem}
\begin{proof}
Corollary~\ref{c:epi} implies \eqref{e:150117a}. 
Using Fact~\ref{f:epi}, we have \eqref{e:map}, which yields 
\begin{equation}
\frac{x_n}{x_{n+1}}=1 +
\frac{f(x_{n+1})-f(0)}{x_{n+1}}x_{n+1}^* \to 1+f'(0)f'(0)=1
\end{equation}
because of $f'(0)=0$ and \cite[Proposition~17.32]{BC2011}. 
This gives \eqref{e:sublinear}.
Now suppose that $f$ is differentiable on $\left[0,\delta\right]$
and \eqref{e:lim} holds. 
Hence, using also \eqref{e:map}, 
\begin{equation}
\frac{x_{n+1} -x_{n+2}}{x_n-x_{n+1}}  
=\frac{\displaystyle
\frac{f(x_{n+2})f'(x_{n+2})}{x_{n+2}^q}}{\displaystyle\frac{f(x_{n+1})f'(x_{n+1})}{x_{n+1}^q}} 
\left( \frac{x_{n+2}}{x_{n+1}} \right)^q
\to\frac{c_q}{c_q}\cdot 1^q =1,
\end{equation}
as claimed. 
\end{proof}

\begin{remark}
The function $f$ satisfies \eqref{e:lim} with $q=2a-1$ and $c_q =
a\varphi^2(0)$ whenever $f(x) =x^a \varphi(x)$, where
$a\in\RR\smallsetminus\{0\}$, $\delta\in\RPP$, $\varphi$ is differentiable
on $[0,\delta]$, $\varphi'$ is continuous at $0$, and $\varphi(0) \ne 0$.
\end{remark}

\begin{proposition}
\label{p:0621}
Suppose that $f$ is differentiable on $[0,\delta]$, where
$\delta\in\RPP$.
Set 
\begin{equation}
\big(\forall q\in\left[1,\pinf\right[\big)\quad
c_q = \lim_{x\downarrow 0} \frac{f(x)f'(x)}{x^q},
\end{equation}  
where $c_q$ is either undefined if the limit does not exist or in
$[0,\pinf]$. 
Let $q\in\left[1,\pinf\right[$. 
Then the following hold:
\begin{enumerate}
\item
\label{p:0621.0}
$c_1=0$. 
If $c_q = 0$, then $c_{q'} =0$ for $1 \leq q' <q$.  
If $c_q >0$, then $c_{q'} =\pinf$ for $q' >q$.
\item
\label{p:0621.1}
If $c_q=0$, then 
\begin{equation}\label{e:cq=0}
\frac{\displaystyle\frac{x_n}{x_{n+1}}-1}{x_{n+1}^{q-1}} \to 0
\end{equation}  
and 
\begin{equation}\label{e:cq=0ab}
(\forall\ve\in\RPP)(\exi m\in\NN)(\forall n\geq m)\quad
x_n \geq 
\begin{cases}
\big( \exp(-\ve)\big)^n, &\text{when } q =1; \\[+4mm]
\displaystyle \frac{1}{\big((q-1)n\ve\big)^{1/(q-1)}}, &\text{when } q >1.
\end{cases}
\end{equation}
\item
\label{p:0621.2}
If $c_q>0$, then $q>1$ and 
$\displaystyle \frac{x_n}{\left( \frac{1}{n} \right)^{1/(q-1)}} 
\to \frac{1}{\big((q-1)c_q\big)^{1/(q-1)}}$. 
\end{enumerate}
\end{proposition}
\begin{proof} \ \ref{p:0621.0}:
Using L'H\^opital's rule, we have
$\lim_{x \downarrow 0} \frac{f(x)}{x} =\lim_{x \downarrow 0}
\frac{f'(x)}{1} =f'(0) =0$ and hence
$c_1 = 
\lim_{x \downarrow 0} \frac{f(x)f'(x)}{x} = 0$. 
The remaining statements follow now readily. 

\ref{p:0621.1}:
It follows from \eqref{e:map} that
\begin{equation}
\frac{\frac{x_n}{x_{n+1}}-1}{x_{n+1}^{q-1}} 
=\frac{f(x_{n+1})f'(x_{n+1})}{x_{n+1}^q} \to c_q =0;
\end{equation}
thus, \eqref{e:cq=0} holds.
Now write 
\begin{equation}
x_{n+1} = 
 x_{n} - \frac{f(x_{n+1})f'(x_{n+1})}{x_{n+1}^q}
\left(\frac{x_{n+1}}{x_n}\right)^qx_n^q
= x_n - \rho_n x_n^q,
\end{equation}
and note that $\rho_n \to 0$
because $x_{n+1}/x_{n}\to 1$ and $c_q=0$. 
Thus, \eqref{e:cq=0ab} holds due to Example~\ref{ex:anotherjon}.

\ref{p:0621.2}: We must have $q>1$ since otherwise $c_q=c_1=0$ by
\ref{p:0621.0}, which is absurd. From \eqref{e:map}, we have
\begin{equation}
\frac{x_n}{x_{n+1}} = 1 + \frac{f(x_{n+1})f'(x_{n+1})}{x_{n+1}}
\to 1
\end{equation}
and also, for every $\nnn$,
\begin{equation}
x_n = x_{n+1} + f(x_{n+1})f'(x_{n+1})
= x_{n+1} + \frac{f(x_{n+1})f'(x_{n+1})}{x_{n+1}^q}x_{n+1}^q.
\end{equation}
The conclusion therefore follows from Example~\ref{ex:scheissruecken}. 
\end{proof}

\begin{example}[$\frac{1}{p}|x|^p$, where $p>1$]
\label{ex:MAPx^p}
Suppose that $f(x) = \tfrac{1}{p}|x|^p$, where $1<p<\pinf$.
Let $x\in\RPP$. 
Then $f(x) = \tfrac{1}{p}x^p$ and $f'(x) = x^{p-1}$. 
Setting $q = 2p-1>1$, we have $c_q = \tfrac{1}{p} >0$, 
and so $x_n \to 0$ logarithmically, using Theorem~\ref{t:MAP}. 
Moreover, by Proposition~\ref{p:0621}\ref{p:0621.2},
\begin{equation}
\label{e:150117b}
\frac{x_n}{\left( \frac{1}{n} \right)^{1/(2p-2)}} 
\to \frac{1}{\left( \frac{2p-2}{p} \right)^{1/(2p-2)}}.
\end{equation}
For a couple of cases, one can actually 
invert \eqref{e:map} and simplify \eqref{e:150117b}:
\begin{equation}
p = \frac{3}{2} 
\quad\Rightarrow\quad
\frac{x_n}{1/n} \to 
\frac{3}{2}
\;\;\text{and}\;\;
x_{n+1} = \frac{\sqrt{9+24x_n}-3}{4};
\end{equation}
and 
\begin{equation}
p=2 
\quad\Rightarrow\quad
\frac{x_n}{1/\sqrt{n}} \to 1 
\;\;\text{and}\;\;
x_{n+1} = 
\frac{1}{3}{\frac { \left( 27\,x_n+3\, \sqrt{81\,{x_n^2}+24}
\right) ^{2/3}-6}{\sqrt [3]{27\,x_n+3\, \sqrt{81\,{x_n^2}+24}}}}.
\end{equation}
\end{example}

\begin{example}[$R -\sqrt{R^2-x^2}$]
\label{ex:mapball}
Suppose that $R\in\RPP$ and that 
$f(x) =R -\sqrt{R^2-x^2}$ on its domain $[-R,R]$. 
Let $\nnn$. 
Then $f'(x) =\frac{x}{\sqrt{R^2-x^2}}$, and by \eqref{e:map},
\begin{equation}
x_n =x_{n+1} +\left(R -\sqrt{R^2-x_{n+1}^2}\right)\frac{x_{n+1}}{\sqrt{R^2-x_{n+1}^2}} 
=\frac{Rx_{n+1}}{\sqrt{R^2-x_{n+1}^2}}.
\end{equation}
It follows that
\begin{equation}
x_{n+1} =\frac{Rx_n}{\sqrt{x_n^2+R^2}},
\end{equation}
and also $\frac{R^2}{x_{n+1}^2} =1 +\frac{R^2}{x_n^2}$. 
Hence, $\frac{R^2}{x_n^2} =n +\frac{R^2}{x_0^2}$, which yields
the \emph{explicit} formula 
\begin{equation}
x_n =\frac{Rx_0}{\sqrt{nx_0^2+R^2}} = \frac{R}{\sqrt{n+(R/x_0)^2}} \sim
\frac{R}{\sqrt{n}},
\end{equation}
which shows that $x_n\to 0$ logarithmically. 
\end{example}

\begin{example}
Suppose that 
$f(x) = \exp(|x|)-|x|-1$. 
Then $f(x)f'(x) = \frac{1}{2}x^3 +\frac{5}{12}x^4
+\frac{5}{24}x^5 + O(x^6)$ 
on $\RPP$. Thus, 
Proposition~\ref{p:0621}\ref{p:0621.2} with
$q=3$ and $c_q =\frac{1}{2}$ yields
$x_n/(1/\sqrt{n}) \to 1$.
\end{example}

\begin{example}
Suppose that $f=\cosh$.
Then $f(x)f'(x)=\frac{1}{2}x^3 + \frac{1}{8}x^5 + O(x^7)$ and
hence again by Proposition~\ref{p:0621}\ref{p:0621.2} with
$q=3$ and $c_q =\frac{1}{2}$ yields
${x_n}/(1/\sqrt{n}) \to 1$.
\end{example}

\begin{example}[extremely slow convergence]
Suppose that 
$f(x) = \exp(-x^{-2})$ with domain $[-\sqrt{2/3},\sqrt{2/3}]$.
Then $c_q = 0$ in Proposition~\ref{p:0621}\ref{p:0621.1} and
hence, according to \eqref{e:cq=0},
\begin{equation}
\frac{\frac{x_n}{x_{n+1}}-1}{x_{n+1}^{q-1}}\to 0
\quad\text{for \emph{every} $q>1$}.
\end{equation}
Furthermore, 
the convergence $(x_n)_\nnn$ to $0$ is
extremely slow in the sense that 
\begin{equation}
\frac{1}{x_n} \leq
O(n^{1/p})\quad\text{for \emph{every} $p>0$.} 
\end{equation}
\end{example}

\begin{remark}[MAP sequence is essentially a PPA sequence]
\label{r:map=ppa}
Note that, by \eqref{e:map},
\begin{equation}
(\forall\nnn)\quad
x_n \in (\Id + f\cdot\partial f)(x_{n+1}) = 
\big(\Id + \partial \tfrac{1}{2}f^2\big)(x_{n+1})
\end{equation}
and so 
\begin{equation}
(\forall\nnn)\quad x_{n+1} = P_{\frac{1}{2}f^2}(x_n).
\end{equation}
Since $f^2$ is convex (by, e.g.,
\cite[Proposition~8.19]{BC2011}), 
we see that the sequence $(a_n)_\nnn = (x_n,0)_\nnn$ generated by
MAP is essentially the same as the sequence generated by PPA for
the function  $\tfrac{1}{2}f^2$!
This useful connection will be further discussed after we recall
a special case of a result due to G\"uler.
\end{remark}

\begin{fact}[G\"uler]
\label{f:Guler}
{\rm (See \cite[Theorem~3.1]{Guler}.)}
Let $H$ be a real Hilbert space,
let $f\colon H\to\RX$ be convex, lower semicontinuous, and
suppose that the sequence $(x_n)_\nnn$ generated by the PPA converges
(strongly) to some minimizer $z$ of $f$. 
Then $f(x_n)-f(z) = o(1/n)$, i.e., 
$n(f(x_n) -f(z)) \to 0$.
\end{fact}

Combining Remark~\ref{r:map=ppa} with
Fact~\ref{f:Guler} results in the following:

\begin{corollary}
The MAP sequence $(a_n)_\nnn = (x_n,0)_\nnn$ satisfies 
\begin{equation}
\label{e:1205a}
n f^2(x_n)\to 0.
\end{equation}
\end{corollary}

\begin{example}[$\frac{1}{p}|x|^p$ revisited]
Suppose that $f(x)=\frac{1}{p}|x|^p$, where $1<p$. 
By Example~\ref{ex:MAPx^p},
\begin{equation}
x_n\sim \left( \frac{1}{n} \right)^{1/(2p-2)}.
\end{equation}
Then \eqref{e:1205a} becomes
\begin{equation}
n f^2(x_n)\sim n x_n^{2p} \sim \frac{1}{n^{1/(p-1)}} \to 0.
\end{equation}
Note that this also shows that this consequence of 
Fact~\ref{f:Guler} is {\em sharp} 
in the sense that it cannot be improved to 
$n^{1+\ve} f^2(x_n)\to 0$, where $\ve>0$. 
(Indeed, if $n^{1+\ve} f^2(x_n)\to 0$, 
then we obtain a contradiction for sufficiently large $p$.)
\end{example}

\section{Douglas--Rachford algorithm (DRA)}

\label{s:dra}

Finally, we investigate the Douglas--Rachford algorithm. 
As in Section~\ref{s:map}, we assume that 
that 
\begin{subequations}\label{e:drafAB}
\begin{empheq}[box=\mybluebox]{equation}
f \colon \RR\to\RX
\quad\text{is convex, lower semicontinuous, and proper,}
\end{empheq}
with 
\begin{empheq}[box=\mybluebox]{equation}
\text{$f$ even, $f(0)=0$, $f>0$ otherwise, and $f'(0)=0$,}
\end{empheq}
and that 
\begin{empheq}[box=\mybluebox]{equation}
A = \RR \times \{0\}
\quad\text{and}\quad
B= \epi f.
\end{empheq}
\end{subequations}
We now turn to the sequence generated by the 
Douglas--Rachford algorithm. 
We assume that 
\begin{subequations}
\label{e:draseq}
\begin{empheq}[box=\mybluebox]{equation}
x_0\in\RPP\cap\dom f,\;\;
r_0=0,\;\;
z_0 = (x_0,0)\in A,\;\;
\end{empheq}
and 
\begin{empheq}[box=\mybluebox]{equation}
(\forall\nnn)\quad
z_{n+1} = Tz_n = (x_{n+1},r_{n+1}),
\end{empheq}
where 
\begin{empheq}[box=\mybluebox]{equation}
\label{e:DR}
T = \Id-P_A + P_BR_A.
\end{empheq}
\end{subequations}
Since  $N_{A-B}(0,0)=\RP(0,1)$, we have 
\begin{equation}
A\cap B = \{(0,0)\}
\quad\text{and}\quad
\Fix T = \RP(0,1),
\end{equation}
Hence we deduce from Fact~\ref{f:dra}, \eqref{e:drafAB} and
\eqref{e:draseq} that 
\begin{equation}
x_n\to 0 \;\;\text{and}\;\;
r_n\to r_\infty \in \RP.
\end{equation}

Let us now investigate the effect of carrying out
one DRA step:

\begin{corollary}[one DRA step]
\label{c:DRstep}
Let $(x,r)\in\RR^2$, 
set $(x_+,r_+)= T(x,r)$, and suppose that 
$0<x\in \dom f$ and $0\leq r<f(x)$. 
Then there exists $x_+^*\in\RR$ such that 
\begin{subequations}
\begin{equation}
\label{e:DRstep}
0<x_+ = x -r_+x_+^*<x,\;\;x_+^*\in \partial f(x_+)
\;\;\text{and}\;\;
r_+ = r+f(x_+)>r
\end{equation}
and 
\begin{equation}
x^2+r^2\geq x_+^2 + (r_+-r)^2 + (x-x_+)^2 + r_+^2.
\end{equation}
\end{subequations}
\end{corollary}
\begin{proof}
First, we note that 
$R_A(x,r) = (x,-r)$. 
Set $(y,s)= P_B(x,-r)$.
By Fact~\ref{f:epi},
\begin{equation}
y = x -(r+f(y))x^*_+\;\;\text{for some}\;\;
x_+^*\in \partial f(y)
\;\;\text{and}\;\;
s = f(y).
\end{equation}
Now, \eqref{e:DR} gives
\begin{equation}
(x_+,r_+) = (\Id-P_A)(x,r) + P_BR_A(x,r) = 
(x,r)-(x,0)+(y,s) = (y,r+s).
\end{equation}
Thus $x_+=y$, 
\begin{equation}
x_+ = x -(r+f(x_+))x_+^*
\;\;\text{and}\;\;
r_+ = r+f(x_+),
\end{equation}
as claimed. 
The rest follows from Corollary~\ref{c:epi}.
\end{proof}

\begin{remark}[DRA step is related to a PPA step]
Consider Corollary~\ref{c:DRstep}. 
Then
\begin{equation}
x_+ = P_{rf + \frac{1}{2}f^2}(x)
\;\;\text{and}\;\;
r_+ = r + f(x_+),
\end{equation}
which reveals a connection between the DRA step and the PPA step 
for $rf+\tfrac{1}{2}f^2$.
\end{remark}

\begin{theorem}[DRA sequence]
\label{t:drseq}
The DRA sequence satisfies
\begin{subequations}
\label{e:DRx_nboth}
\begin{equation}
x_n\downarrow 0
\quad\text{and}\quad
r_n\uparrow r_\infty\in\RPP,
\end{equation}
and for every $\nnn$,
there exits $x_{n+1}^*\in\partial f (x_{n+1})$ such that 
\begin{equation}
\label{e:DRx_n}
0<x_{n+1} = x_n - r_{n+1}x_{n+1}^*<x_n 
\quad \text{and} \quad
r_{n+1}=r_n + f(x_{n+1}).
\end{equation}
\end{subequations}
Now suppose that furthermore that $f''_+(0)$ exists in $[0,\pinf]$. 
Then 
\begin{equation}
\label{e:DRlim}
\frac{x_{n+1}}{x_n} =\frac{1}{1+r_{n+1}\frac{x_{n+1}^*}{x_{n+1}}} \to \frac{1}{1+r_\infty f''_+(0)}
\end{equation}
and exactly one of the following holds:
\begin{enumerate}
\item 
\label{t:drseq1}
$f''_+(0)=\pinf$ and 
$x_n\to 0$ superlinearly. 
\item 
\label{t:drseq2}
$f''_+(0) \in\RPP$ and 
$x_n\to 0$ linearly.
\item 
\label{t:drseq3}
$f''_+(0) =0$ and 
$x_n\to 0$ sublinearly.
If there exists $q\in\RR$ such that
\begin{equation}\label{e:lim2}
\lim_{x\downarrow 0} \frac{f'(x)}{x^q} = c\in\RPP, 
\end{equation}
then $x_n\to 0$ logarithmically; moreover,
if additionally $q>1$, then 
\begin{equation}
\label{e:DRsub}
%q>1 \quad\Rightarrow\quad
\frac{x_n}{\left( \frac{1}{n} \right)^{1/(q-1)}} \to \frac{1}{\left( (q-1)r_\infty c \right)^{1/(q-1)}}.
\end{equation}
\end{enumerate}
\end{theorem}
\begin{proof}
\eqref{e:DRx_nboth} follows from Corollary~\ref{c:DRstep}. 
Divide \eqref{e:DRx_n} by $x_{n+1}$, solve for $x_n/x_{n+1}$,
then take reciprocals to obtain 
\begin{equation}
\frac{x_{n+1}}{x_n} =\frac{1}{1+r_{n+1}\frac{x_{n+1}^*}{x_{n+1}}}.
\end{equation}
Now assume that  $f''_+(0)$ exists; it belongs to $[0,\pinf]$
because $x>0$ $\Rightarrow$ $f'(x)>0$. 
Since $x_n \downarrow 0$,
we see that 
\begin{equation}
\frac{x_{n+1}^*}{x_{n+1}} =\frac{f'(x_{n+1})}{x_{n+1}} 
=\frac{f'(x_{n+1})-f'(0)}{x_{n+1}-0} \to f''_+(0).
\end{equation}
Altogether, we get \eqref{e:DRlim}. 
Items \ref{t:drseq1} and \ref{t:drseq2} are now clear,
so let us focus on \ref{t:drseq3}.
Obviously $x_n\to 0$ sublinearly if and only if $f''_+(0)=0$
which we henceforth assume, along with \eqref{e:lim2}. 
It follows from \eqref{e:DRx_n} that
\begin{equation}
\frac{x_{n+1} -x_{n+2}}{x_n-x_{n+1}} 
=\frac{r_{n+2}f'(x_{n+2})}{r_{n+1}f'(x_{n+1})}
=\frac{r_{n+2}}{r_{n+1}}\frac{\frac{f'(x_{n+2})}{x_{n+2}^q}}{\frac{f'(x_{n+1})}{x_{n+1}^q}}
\left( \frac{x_{n+2}}{x_{n+1}} \right)^q
\to \frac{r_\infty}{r_\infty}\cdot \frac{c}{c}\cdot 1^q =1;
\end{equation}
hence, $x_n \to 0$ logarithmically.
Finally assume that $q >1$. 
Writing (see \eqref{e:DRx_n}) 
\begin{equation}
x_n =x_{n+1} +r_{n+1}\frac{f'(x_{n+1})}{x_{n+1}^q}x_{n+1}^q 
= x_{n+1} +\delta_n x_{n+1}^q,
\end{equation}
where $\delta_n \to r_\infty c>0$, 
we obtain \eqref{e:DRsub} through Example~\ref{ex:scheissruecken}.
\end{proof}

\begin{example}[$\frac{1}{p}|x|^p$, where $1<p<\pinf$]
\label{ex:DRx^p}
Suppose that 
$f(x)=\frac{1}{p}|x|^p$, where $1<p<\pinf$. 
Then exactly one of the following holds:
\begin{enumerate}
\item 
\label{ex:DRx^p1}
$1<p<2$ \;and\;
$\displaystyle \frac{x_{n+1}}{x_n^{1/(p-1)}} \to
\frac{1}{r_\infty^{1/(p-1)}} >0.$
\item
\label{ex:DRx^p2}
$p=2$ and $x_n\to 0$ linearly with rate $1/(1+r_\infty)$. 
\item
\label{ex:DRx^p3}
$2<p<\pinf$, $x_n\to 0$ logarithmically, and 
$\displaystyle
\frac{x_n}{\left( \frac{1}{n} \right)^{1/(p-2)}} \to
\frac{1}{\left( (p-2)r_\infty\right)^{1/(p-2)}}.
$
\end{enumerate} 
\end{example}
\begin{proof}\ \ref{ex:DRx^p1}:
From \eqref{e:DRx_n}, we obtain
\begin{equation}
\label{e:DRx^p}
(\forall\nnn)\quad 
x_n =x_{n+1} +r_{n+1}x_{n+1}^{p-1}.
\end{equation}
Since $x_n \downarrow 0$ and $p<2$,
we have 
\begin{equation}
(\forall \ve >0)(\exi m \in \NN)(\forall n\geq m) \quad 0 < x_{n+1} <\ve x_{n+1}^{p-1}.
\end{equation}  
Combining yields 
$r_{n+1}x_{n+1}^{p-1} <x_n <(r_{n+1} +\ve)x_{n+1}^{p-1}$. 
In turn, 
$\frac{x_n}{x_{n+1}^{p-1}} \to r_\infty >0$ and
the conclusion follows.

\ref{ex:DRx^p2}: Clear from \eqref{e:DRlim}. 

\ref{ex:DRx^p3}: Apply Theorem~\ref{t:drseq}\ref{t:drseq3} with
$q=p-1>1$ and $c=1$. 
\end{proof}

\begin{remark}[comparison MAP vs DR when $f(x) =\frac{1}{p}|x|^p$]
Suppose $f(x) =\frac{1}{p}|x|^p$, where $1<p<\pinf$.
According to Example~\ref{ex:MAPx^p}, 
the MAP sequence $(x_n)_\nnn$ exhibits logarithmic convergence to
$0$ and 
\begin{equation} 
x_n\sim \left( \frac{1}{n} \right)^{1/(2p-2)}.
\end{equation}
On the other hand, Example~\ref{ex:DRx^p} yields the following
for the DRA sequence $(x_n)_\nnn$: 
\bigskip
\begin{table*}[H] \centering
\begin{tabular}{@{}r|l@{}} \toprule
$p$ & convergence of the DRA sequence $(x_n)_\nnn$ to $0$\\ \midrule
$1<p<2$ & superlinear with order $\displaystyle \frac{1}{p-1}$ \\[+2mm]
$p=2$ & linear with rate $\displaystyle \frac{1}{1+r_\infty}$ \\[+2mm]
$2<p<\pinf$& logarithmic and $\displaystyle x_n\sim \left( \frac{1}{n}
\right)^{1/(p-2)}$\\
\bottomrule
 \end{tabular}
\end{table*}
We conclude that \emph{in all cases}, 
the DRA sequence converges to $0$ \ {\em faster} than
the MAP sequence\footnote{It is interesting to note that
DRA performs better than MAP when $A$ and $B$ are two subspaces
with a small Friedrichs angle (see \cite[Section~8]{BBNPW}).}.
To illustrate this, set $x_0=1$. 
Letting the parameter $p$ range from $1$ to $3$, 
we show in Figure~\ref{fg:1} the first 100 terms of the MAP
sequence $(x_n)_\nnn$ and of the DRA sequence $(x_n)_\nnn$. 
\begin{figure}[H]
  \centering
  \begin{subfigure}{.45\textwidth}
    \centering
    \includegraphics[width=2.8in]{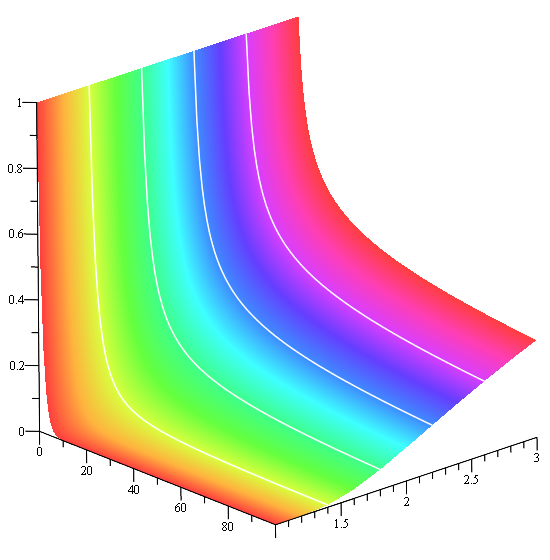}
    \caption{MAP sequence $(x_n)_\nnn$}
    \label{fg:1a}
  \end{subfigure}%
  \hspace*{.1\textwidth}
  \begin{subfigure}{.45\textwidth}
    \centering
    \includegraphics[width=2.8in]{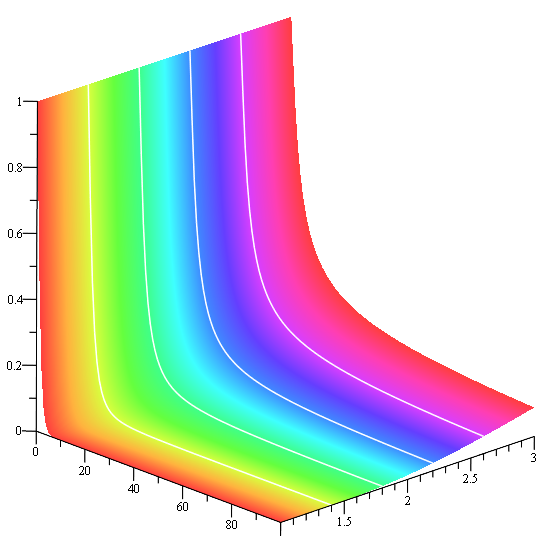}
    \caption{DRA sequence $(x_n)_\nnn$}
    \label{fg:1b}
  \end{subfigure}
\caption{The distance of the first 100 terms to the solution for $p\in[1,3]$}
\label{fg:1}
\end{figure}
\noindent
Although both sequences converge to $0$, 
the solution, the stark contrast in their
speed of convergence is shown in 
Figure~\ref{fg:2} where we plot the quotient sequence of the MAP
sequence divided by the DRA sequence. As predicted by the theory,
the terms tend to $\pinf$ when $1<p<2$ illustrating the much
faster convergence of the DRA sequence. 

\begin{figure}[H]
\centering
\includegraphics[width=4in]{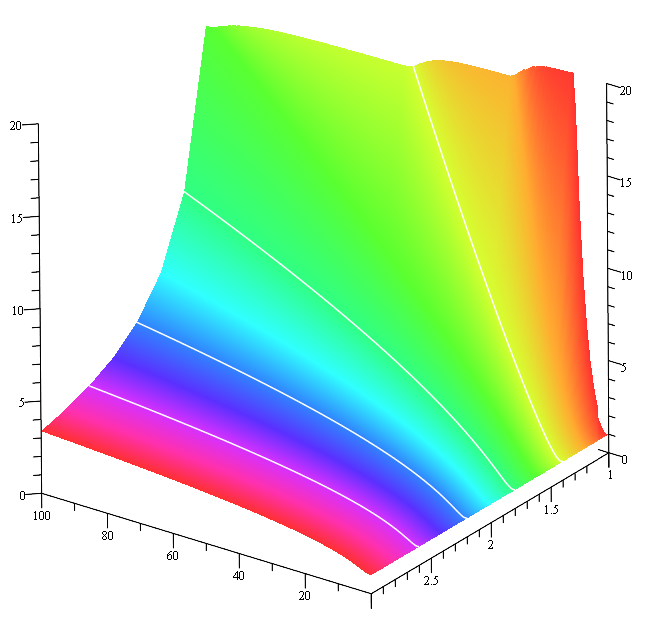}
\caption{The MAP sequence divided by the DRA sequence}
\label{fg:2}
\end{figure}
\end{remark}

\begin{example}[comparison MAP vs DR when $f(x)=R -\sqrt{R^2-x^2}$]
\ \\
Suppose that $R\in\RPP$ and that 
 $f(x) =R -\sqrt{R^2-x^2}$ on its domain $[-R,R]$. 
 According to Example~\ref{ex:mapball}, the
 MAP sequence $(x_n)_\nnn$ exhibits logarithmic convergence; in
 fact,
\begin{equation}
x_n \sim \frac{R}{\sqrt{n}}.
\end{equation}
We now turn to the DRA sequence $(x_n)_\nnn$.
By \eqref{e:DRx_n}, we have for every $\nnn$
\begin{equation}
x_n =x_{n+1} +\left(r_n +R
-\sqrt{R^2-x_{n+1}^2}\right)\frac{x_{n+1}}{\sqrt{R^2-x_{n+1}^2}}
=(r_n +R)\frac{x_{n+1}}{\sqrt{R^2-x_{n+1}^2}};
\end{equation} 
consequently, 
\begin{equation}
x_{n+1} =\frac{Rx_n}{\sqrt{(r_n+R)^2 +x_n^2}}.
\end{equation}
%where $r_{n+1} =r_n +1 -\sqrt{1-x_{n+1}^2}$. 
Since $f''(x) = R^2/(R^2-x^2)^{3/2}$, we
have $f''(0) = 1/R\in\RPP$ and
therefore, by \eqref{e:DRlim}, 
the DRA sequence 
\begin{equation}
\text{ $x_n \to 0$\; linearly }
\end{equation}
with rate $1/(1+r_\infty/R)$.
Once again, the DRA sequence converges much faster than the MAP
sequence!
\end{example}

Let us conclude. 
The results in this paper suggest that, for the convex feasibility problem, 
DRA outperforms MAP in cases of ``bad geometry'' 
(such as the absence of constraint qualifications or a ``zero angle''
between the constraints at the intersection). 
Since our proof techniques do not naturally generalize, 
it would be interesting to study these questions in higher-dimensional
space and other classes of convex sets. 

\section*{Acknowledgments}
HHB was partially supported by the Natural Sciences and
Engineering Research Council of Canada 
and by the Canada Research Chair Program.
He also acknowledges the hospitality and the support of 
Universit\'e de Toulouse, France
during the preparation of an early draft of this paper. 
MND was partially supported by an NSERC accelerator grant of HHB. 
HMP was partially supported by an internal grant of
University of Massachusetts Lowell.

%\small

\begin{appendices}
\section{}
\label{a:Stolz}
\begin{proof}[Proof of Fact~\ref{f:genStolz}]
(This proof is taken from 
\url{http://www.imomath.com/index.php?options=686} and included
here for completeness as we were not able to locate a book or journal
reference.)
The second inequality is obvious.
We only prove the right inequality since the proof of the left inequality
is similar. 
Without loss of generality, we assume that $b_n\to+\infty$ and that 
$\lambda = \varlimsup_\ntoinf \frac{a_{n+1}-a_n}{b_{n+1}-b_n}<\pinf$.
Let $\gamma\in\left]\lambda,\pinf\right[$. 
Then there exists $m\in\NN$ such that 
$(\forall n\geq m)$ $ \frac{a_{n+1}-a_n}{b_{n+1}-b_n}<\gamma$, 
i.e., $a_{n+1}-a_n < \gamma(b_{n+1}-b_n)$.
Let $n>m$.  Then $a_n-a_m < \gamma(b_n-b_m)$ and hence 
\begin{equation}
\frac{a_n}{b_n} - \frac{a_m}{b_n} < \gamma -
\gamma\frac{b_m}{b_n}.
\end{equation}
Taking $\varlimsup_{n\to\infty}$ yields
\begin{equation}
\varlimsup_\ntoinf \frac{a_n}{b_n} \leq \gamma.
\end{equation}
Now let $\gamma \downarrow \lambda$ to complete the proof.
\end{proof}

\section{}
\label{a:Guler}
\begin{proof}[Proof of Fact~\ref{f:Guler}]
(This proof is a special case taken from \cite{Guler} and included here for completeness.)
Denote the minimizers of $f$ by $Z$.
Let $x\in H$,
set $p=P_fx$, $p^*=x-p\in\partial f(p)$,
and assume that $p\notin Z$. 
We have
\begin{equation}\label{e:guler-a}
(\forall y\in H)\quad
f(y)-f(p)\geq\scal{p^*}{y-p}.
\end{equation}
Hence, for every $z\in Z$,
\begin{subequations}
\begin{align}
f(z)-f(p) &\geq \scal{p^*}{z-p}
= \scal{p^*}{z-x} + \scal{p^*}{x-p}\\
&\geq -\|p^*\|\|z-x\|+\|p^*\|^2
\geq -\|p^*\|\|z-x\|.
\end{align}
\end{subequations}
Since $f(p)\geq f(z)$, we learn that
\begin{equation}
\|p^*\|\geq\frac{f(p)-f(z)}{\|x-z\|}\geq 0.
\end{equation}
Setting $y=x$ in \eqref{e:guler-a}, we have
\begin{equation}
f(x)-f(p) \geq \scal{p^*}{x-p}=\|p^*\|^2.
\end{equation}
Therefore,
\begin{align}
\big(f(x)-f(z)\big) - \big(f(p)-f(z)\big) 
= f(x)-f(p)
\geq \|p^*\|^2
\geq \left( \frac{f(p)-f(z)}{\|x-z\|}\right)^2.
\end{align}
It follows that
\begin{equation}
f(x)-f(z) \geq 
\big( f(p)-f(z)\big)
\left(1 +  \frac{f(p)-f(z)}{\|x-z\|^2}\right);
\end{equation}
equivalently,
\begin{equation}
\frac{1}{f(x)-f(z)} 
\leq \frac{1}{ f(p)-f(z)}
\left(1 +  \frac{f(p)-f(z)}{\|x-z\|^2}\right)^{-1}.
\end{equation}
On the other hand, the definition of the proximal mapping
yields
\begin{equation}
f(p) \leq f(p) + \tfrac{1}{2}\|x-p\|^2
\leq f(z) + \tfrac{1}{2}\|x-z\|^2,
\end{equation}
which implies
\begin{equation}
\frac{f(p)-f(z)}{\|x-z\|^2} \leq \frac{1}{2}.
\end{equation}
Consider the function
$\alpha(t) = \frac{1}{1+t}$ with domain $\left]-1,\pinf\right[$. 
Clearly, $\alpha$ is convex, $\alpha(0) = 1$ and 
$\alpha(\frac{1}{2}) = \frac{2}{3}$.
The line described by $t\mapsto 1-\frac{2}{3}t$ 
goes through the same points and lies between these points
above the graph of $\alpha(\cdot)$ (by convexity). 
Hence 
\begin{equation}
\big(\forall 0\leq t \leq \tfrac{1}{2}\big)\quad
\frac{1}{1+t} \leq 1-\frac{2}{3}t.
\end{equation}
Altogether, we deduce that 
\begin{subequations}
\begin{align}
\varphi(x) &= \frac{1}{f(x)-f(z)} \\
&\leq 
\frac{1}{ f(p)-f(z)}
\left(1 - \frac{2}{3}\cdot \frac{f(p)-f(z)}{\|x-z\|^2} \right)\\
&= \varphi(p) - \frac{2}{3\|x-z\|^2}.
\end{align}
\end{subequations}
Now assume that $z\leftarrow x_{n+1} = P_fx_n$.
Then $\varphi(x_{n+1})-\varphi(x_n)\geq \frac{2}{3}\|x_n-z\|^{-2}\to\pinf$.
By Corollary~\ref{c:Cesaro}, 
$\lim\frac{\varphi(x_n)}{n} = \pinf$.
Now take reciprocals. 
\end{proof}

\end{appendices}

\end{document}